\documentclass[letterpaper, 10 pt, conference]{ieeeconf}  
\pdfoutput=1                                              

\IEEEoverridecommandlockouts                              
\usepackage{amsmath}
\usepackage{amssymb,amsfonts}
\usepackage{algorithmicx}
\usepackage{algorithm}
\usepackage{algpseudocode}
\usepackage{booktabs}
\usepackage{times,color,url}
\usepackage{comment}
\usepackage{subcaption}

\usepackage{url}
\usepackage{cite}
\usepackage{hyperref}
\usepackage{graphicx}
\usepackage{setspace}

\graphicspath{ {./images/} }
\usepackage[outdir=./images/]{epstopdf}

\newcommand{\setR}{\mathbb{R}}
\newcommand{\setRnex}{\overline{\setR}^n}
\newcommand{\setN}{\mathbb{N}}
\newcommand{\setC}{\mathbb{C}}
\newcommand{\agents}{\mathcal{N}}

\newcommand{\tfwy}[1]{\Phi_{y #1}}
\newcommand{\tfwx}[1]{\Phi_{x #1}}
\newcommand{\tfwu}[1]{\Phi_{u #1}}

\newcommand{\graphcom}{\mathcal{G}}
\newcommand{\edges}{\mathcal{E}}
\newcommand{\feas}{\mathcal{X}}
\newcommand{\Vector}[1]{\begin{bmatrix}
		#1
\end{bmatrix}}

\DeclareMathOperator{\prox}{Prox}

\DeclareMathOperator*{\argmin}{\arg \min}

\newcommand{\proxim}[1]{\prox_{#1}}

\newcommand{\indic}{\mathcal{I}}

\newcommand{\lagrangian}{\mathcal{L}}
\newcommand{\conj}[1]{{#1}^{\star}}

\newcommand{\norm}[1]{\lVert #1 \rVert}

\newcommand{\grad}{\nabla}

\newcommand{\fir}{\mathcal{J}}

\makeatletter
\newcommand{\ALC@uniqueautorefname}{Step}
\makeatother
\newtheorem{assumption}{Assumption}
\newtheorem{remark}{Remark}
\newtheorem{lemma}{Lemma}
\newtheorem{theorem}{Theorem}

\allowdisplaybreaks

\title{\LARGE \bf Distributed and Constrained \( \mathcal{H}_2 \) Control Design via System Level Synthesis and Dual Consensus
	ADMM
}

\author{Panagiotis D. Grontas, Michael W. Fisher and Florian D\"orfler
\thanks{This scientific paper was supported by the Onassis Foundation - Scholarship ID: F ZQ 019-1/2020-2021.}
\thanks{Panagiotis D. Grontas and Florian D\"orfler are with the Automatic Control Laboratory, ETH Zurich, 8092 Zurich, Switzerland.}
\thanks{Michael W. Fisher is with the University of Waterloo, Waterloo, Ontario, Canada.}
}
\usepackage{tikz}
\newcommand\copyrighttext{%
	\footnotesize \textcopyright 
	2022 IEEE.  Personal use of this material is permitted.  Permission from IEEE must be obtained for all other uses, in any current or future media, including reprinting/republishing this material for advertising or promotional purposes, creating new collective works, for resale or redistribution to servers or lists, or reuse of any copyrighted component of this work in other works.
	}
\newcommand\copyrightnotice{%
	\begin{tikzpicture}[remember picture,overlay]
		\node[anchor=north,yshift=-15pt] at (current page.north) {\fbox{\parbox{\dimexpr\textwidth-\fboxsep-\fboxrule\relax}{\copyrighttext}}};
	\end{tikzpicture}%
}

\begin{document}


\maketitle
\thispagestyle{empty}
\pagestyle{empty}

\begin{abstract}
	
  Design of optimal distributed linear feedback controllers to achieve a desired aggregate behavior, 
  while simultaneously satisfying state and input constraints, 
  is a challenging but important problem in many applications.
  System level synthesis is a recent technique which has been 
  used to reparametrize the optimal control problem as a convex program.
  Prior work on system level synthesis with state and input constraints has included closed-loop finite impulse 
  response and locality constraints or, in the case where these constraints were lifted using
  a simple pole approximation, only a centralized design was considered.  
  However, closed-loop finite impulse response and locality constraints cannot be satisfied in many applications.  
  Furthermore, the centralized design using the simple pole approximation lacks robustness to communication failures
  and disturbances, has high computational cost and does not preserve data privacy of local controllers.
  The main contribution of this work is to develop a distributed solution to system level synthesis with the 
  simple pole approximation in order to incorporate state and input constraints without closed-loop finite impulse response 
  or locality constraints, and in a distributed implementation.
  To achieve this, it is first shown that the dual of this problem 
  is a distributed consensus problem.
  Then, an algorithm is developed based on the alternating direction method of multipliers
  to solve the dual while recovering a primal solution, and a convergence certificate is provided.
  Finally, the method's performance is demonstrated on a test case of control design for 
  distributed energy resources that collectively provide stability services to the power grid.
\end{abstract}

\copyrightnotice
\section{Introduction}
Optimal design of linear feedback controllers with state and input constraints is a challenging but important problem.
One celebrated approach is the Youla parametrization \cite{youla_param} which casts
the optimal control problem as a convex
program in terms of the Youla parameter,
while input-output parametrization (IOP) \cite{iop} is a more recent method that focuses on output feedback.
Along this direction, the recent work \cite{sls} introduced system level synthesis (SLS)
whereby controllers are parametrized in terms of the closed-loop system responses.
The resulting optimization problem is convex yet infinite-dimensional, hence intractable in general.
To overcome this in the setting with state and input constraints, the authors employ a finite impulse response (FIR) approximation
of the closed-loop transfer functions \cite{chen2019slsconstraints}.
Unfortunately, this approximation is not feasible for stabilizable but uncontrollable systems and, 
even when feasible, gives rise to deadbeat control which
suffers a number of shortcomings such as high computational cost and
lack of robustness to uncertainty and disturbances due to its large control gains
\cite{dbc_shortcomings}.
Furthermore, \cite{sls}, \cite{chen2019slsconstraints} include additional closed-loop 
locality constraints, which are helpful for distributed implementations of SLS, but are not 
satisfied in many networked systems with coupling throughout the different areas in the network, such as power systems.
By contrast, in \cite{msls_placeholder} a computationally-efficient simple pole approximation (SPA) is developed
for which the closed-loop is not FIR, there are no closed-loop locality constraints, suboptimality certificates are derived,
feasibility for stabilizable plants is guaranteed
and prior knowledge of the system's optimal poles can be integrated.
This approach does not result in deadbeat control since the system responses are not FIR.
Moreover, state and input constraints can be non-conservatively incorporated in the SPA formulation
as illustrated in \cite{spa_dvpp_placeholder}.

The goal of this work is to enable a distributed implementation of SLS with SPA.
This offers a number of advantages over the centralized approach, including robustness to communication failures, uncertainty and 
disturbances, scalability due to the distribution
of the computational burden and preservation of data privacy of agents
owing to the absence of a central coordinator.
Achieving this distributed deployment ultimately
amounts to solving a distributed optimization problem.
In particular, the considered problem structure of SPA subject to a peer-to-peer communication setup
belongs to the recently-introduced Distributed Aggregative Optimization (DAO) framework \cite{dao_main}.
Existing DAO algorithms typically require differentiability and strong convexity of the objective 
function (see \cite{dao_main}, \cite{dao_2})
which is not guaranteed to hold for SPA control design.
Although \cite{dao_3} relaxes the strong convexity assumption, 
it requires vanishing step sizes, which substantially reduces the convergence rate.
Moreover, these methods do not have the ability to incorporate constraints coupled across 
multiple devices, which can often appear in practice (e.g.\ see \autoref{sec:test_case}).

The main contribution of this work lies in the development of a new optimal linear feedback 
control design method with non-conservative state and input constraints, no closed-loop FIR or 
locality constraints, and a distributed and convex implementation.  
To do so, starting from the centralized formulation of SLS with SPA we show that the dual
is a consensus problem and, thus, can be solved using the distributed algorithm proposed in \cite{consensus_admm}, 
which is based on the alternating direction method of multipliers (ADMM) \cite{admm_classic}.
This dual consensus ADMM approach was employed in \cite{dual_consensus_zero} and \cite{dual_consensus_large}
to solve decentralized resource allocation problems.
Here, a general distributed optimization scheme is developed, which is an extension of these prior methods
to a larger class of objective functions.
Under weak assumptions, convergence certificates are provided
that guarantee convergence of the primal variables to the set of primal solutions,
which is a stronger convergence result than in prior work \cite{dual_consensus_zero,dual_consensus_large},
where either stronger assumptions are made (such as smoothness, strong convexity, and full rank of associated matrices), 
or it is only shown that each limit point is a minimizer of the primal problem, 
but cannot guarantee the existence of any limit points (so the primal variables could diverge towards infinity).
Then, the algorithm is specialized for the SPA control design by exploiting the underlying problem structure
to simplify the ADMM subproblems thus improving computational efficiency. 

A recent application of distributed control is heterogeneous ensemble control \cite{verena_dvpp},
where the goal is to design local controllers so that in 
aggregate they achieve a desired dynamic behavior as well as possible, subject to device limitations 
and coupling constraints.  
Prior distributed solution methods for this problem rely primarily on heuristics for disaggregating 
desired behavior among the agents, cannot explicitly incorporate state and input constraints, 
and require manual tuning of controller parameters \cite{verena_dvpp, bjork_dvpp}.  
The distributed control design developed here has bounded suboptimality \cite{msls_placeholder}, 
explicitly includes state, input and coupling constraints non-conservatively, 
and does not require any manual tuning related to controller specifications, addressing the limitations of prior work.  
Its effectiveness for solving heterogeneous ensemble control is demonstrated in \autoref{sec:test_case} 
for control of distributed energy resources (DERs) to collectively provide frequency
control to the power grid.

The remainder of this paper is structured as follows.
In \autoref{sec:problem_formulation}, we review SLS and SPA,
and formulate the optimal control design problem.
In \autoref{sec:dual_consensus}, we abstract the previous problem, develop a distributed
ADMM-based algorithm to solve it and establish its convergence.
\autoref{sec:test_case} explains our simulation setup and demonstrates
our algorithm's performance,
while \autoref{sec:conclusion} concludes the paper.

\textit{Notation:} Let \( \setN, \setR, \setC \) respectively denote
the set of natural, real and complex numbers, while
\( \overline{\setR} = \setR \cup \{+\infty\} \).
\( \setR^n \) is the \( n \)-dimensional Euclidean space
and \( \norm{\cdot} \) its norm, and
let \( \setRnex = \setR^n \cup \{ \infty \} \) denote
the extended Euclidean space, also known as the Riemann sphere.
Given a matrix \( A \) we denote \( A^T \) its transpose,
and for some matrix \( B \) of appropriate dimensions
\( [A;B] \) is their vertical concatenation.
A function \( f : \setR^n \to \overline{\setR} \) is \textit{proper} and \textit{closed} 
if its epigraph is, respectively, non-empty and closed.
Let \( \Gamma_n \) denote the set of proper, closed and convex functions
\( f: \setR^n \to \overline{\setR} \).
Then, for \( f \in \Gamma_n \) the \textit{conjugate} of \( f \) is defined as
\( \conj{f}(y) := \sup_x \{ y^T x - f(x) \} \) and satisfies \( \conj{f} \in \Gamma_n \),
the \textit{proximal operator} is 
\( \proxim{f}^{\rho}(x) := \argmin_y \{f(y) + \frac{\rho}{2} \norm{x - y}_2^2\} \)
for any \( \rho > 0 \),
and the \textit{subdifferential} is the set-valued map
\( \partial f(x) := \{ u \in \setR^n ~|~ (\forall y \in \setR^n) ~ f(y) \geq f(x) + u^T (y - x) \} \).
We let \( \indic_{\mathcal{X}} \) denote the \textit{indicator function} of 
\( \mathcal{X} \subseteq \setR^n \) for which it holds\ 
\( \indic_{\mathcal{X}}(x) = 0 \) for \( x \in \mathcal{X} \),
and \( \indic_{\mathcal{X}}(x) = + \infty \) otherwise.
For any sequence \( (x^k)_{k \in \setN} \subseteq \setRnex \) let \( \omega(x^k) \)
denote the \( \omega \) limit set: the set of points \( y \in \setRnex \) such that
there exists a subsequence of \( (x^k)_{k \in \setN} \) which converges to \( y \).
For any set \( \mathcal{S} \subseteq \setRnex \), we say that \( (x^k)_{k \in \setN} \)
converges to \( \mathcal{S} \), denoted by \( x^k \to \mathcal{S} \), if for every open neighborhood \( V \)
of \( \mathcal{S} \), there exists \( N > 0  \) finite such that \( k \geq N \) implies \( x^k \in V \).

\section{Problem Formulation and Background} \label{sec:problem_formulation}
We consider a collection of \( \agents = \{1, \ldots, N\} \) discrete-time LTI systems, 
referred to as agents, with local dynamics
\begin{equation}
	\begin{aligned}
		x_i^{k+1} & = A_i x_i^k + B_i u_i^k + \hat{B}_i w^k \\
		y_i^k & = C_i x_i^k
	\end{aligned}
\end{equation}
for each agent \( i \in \agents \) and time step \( k \in \setN \).
We denote \( x_i^k \in \setR^{n_{x,i}} \) the state and
\( u_i^k \in \setR^{n_{u,i}} \) the individual control signal of each system, 
\( w^k \in \setR^{n_{w}} \) is an external disturbance,
while \( y_i^k \in \setR^{n_{y,i}} \) is the output.
We endow each agent
with a dynamic state feedback controller of the form \( U_i(z) = K_i(z) X_i(z) \),
where \( X_i(z)\) and \( U_i(z) \) are the z-transforms of the signals
\( x_i^k \) and \( u_i^k\), respectively.
Then, the agent-specific closed loop transfer function mapping disturbance
to output is
\( 	\tfwy{,i}(z) = C_i (zI - A_i - B_i K_i(z))^{-1} \hat{B}_i \)
and \( \tfwx{,i}(z), \tfwu{,i}(z) \) are defined similarly,
while we let \( \Phi_{\text{des}}(z) \) be a desired transfer function,
i.e., a design choice.
Our goal is to solve the following model matching problem
\begin{equation} \label{eq:sls_initial}
	\begin{alignedat}{2}
		& \underset{K_1(z), \ldots, K_N(z)}{\textrm{minimize }}~ &&  
		\Big\lVert \sum_{i \in \agents}\tfwy{,i}(z) - \Phi_{\text{des}}(z) \Big\rVert_{\mathcal{H}_2}^2 \\
		& ~~\textrm{subject to} && \tfwx{,i}(z), \tfwu{,i}(z) \in \mathcal{R}
	\end{alignedat}
\end{equation}
where \( \norm{\cdot}_{\mathcal{H}_2} \) is the \( \mathcal{H}_2 \) norm and
\( \mathcal{R} \) is the Hardy space of real, rational, strictly proper
and stable transfer functions, that additionally satisfy time-dependent constraints on
states and inputs, for a collection of known disturbance signals
(which can include, e.g., impulses and/or steps).
Unfortunately, \eqref{eq:sls_initial} is non-convex in
\( K_i(z) \).
In their seminal work \cite{sls}, the authors propose SLS
whereby the previous problem is reformulated by using \( \tfwx{,i}(z), \tfwu{,i}(z) \)
as design variables.
This reparametrization renders the problem convex,
at the price of imposing additional affine constraints.
The resulting controller for each agent can be recovered as \( K_i(z) = \tfwu{,i}(z) \tfwx{,i}^{-1}(z)  \),
although realizations exist that do not require any transfer function inversion.
Crucially, controller recovery relies solely on local information, i.e., \(  \tfwx{,i}(z), \tfwu{,i}(z)  \),
so no additional communication is necessary for this step.
Nonetheless, the problem is infinite-dimensional, and hence intractable in this form.  
A FIR approximation is proposed in \cite{chen2019slsconstraints} to obtain a tractable control design problem, 
but it suffers from the drawbacks discussed in the introduction.

In order to address the limitations of the closed-loop FIR and locality constraints, \cite{msls_placeholder} proposed
an approximation of \( \tfwx{,i}(z), \tfwu{,i}(z) \) using simple poles as follows
\begin{align}
	\tfwx{,i}(z) = \sum_{p \in \mathcal{P}_i}^{} G_{p,i} \frac{1}{z - p}, ~
	\tfwu{,i}(z) = \sum_{p \in \mathcal{P}_i}^{} H_{p,i} \frac{1}{z - p}
\end{align}
where \( \mathcal{P}_i \subseteq \setC \) is a fixed 
finite set of complex poles inside the unit disk and closed under complex conjugation, 
and the decision variables \( G_{p,i}, H_{p,i} \) are complex matrix coefficients associated with 
pole \( p \in \mathcal{P}_i \) and agent \( i \in \agents \).
The authors prove convergence to a globally optimal solution as the number of poles increases,
and suboptimality bounds based on the geometry of the pole selection are provided \cite{msls_placeholder}.
This approximation renders \eqref{eq:sls_initial} both convex and tractable.

Following \cite{msls_placeholder}, 
we denote the impulse response of \( \Phi_{x,i} \) and \( \Phi_{u,i} \)
at time step \( k \) as 
\( \fir^k [\Phi_{x,i}] := \sum_{p \in \mathcal{P}_i} G_{p,i} p^{k-1} \)
and 
\( \fir^k [\Phi_{u,i}] := \sum_{p \in \mathcal{P}_i} H_{p,i} p^{k-1} \),
respectively.
Then, to solve problem \eqref{eq:sls_initial} we recognize that 
the \( \mathcal{H}_2 \) norm is well-approximated by the Frobenius norm of the (sufficiently large) 
finite-horizon impulse response \cite{msls_placeholder}.
Moreover, for a known disturbance signal \( (w^k)_{k \in \setN} \) the state
of system \( i \in \agents \) at step \( k \in \setN \) is given by
\( x_i^k = \sum_{l=0}^{k} \fir^{k-l} [\Phi_{x,i}] \hat{B}_i w^l \),
while the input is \( u_i^k = \sum_{l=0}^{k} \fir^{k-l} [\Phi_{u,i}] \hat{B}_i w^l \).
Crucially, observe that \( \fir^k [\Phi_{x,i}] \) (resp.\ \( \fir^k [\Phi_{u,i}] \)) depends linearly on 
the decision variable \( G_{p,i} \) (resp.\ \( H_{p,i} \)) and hence the constraint 
\( x_i^k \in \mathcal{C} \) (resp.\ \( u_i^k \in \mathcal{C} \))
is convex for any convex set \( \mathcal{C} \).
In fact, we may impose this constraint for a collection of disturbance signals while retaining convexity.
The last step in our problem formulation is to convert the decision variables from complex matrices to real vectors
by representing the real and imaginary parts of \( G_{p,i} \) and \( H_{p,i} \) as separate matrices
and then employing vectorization.
The resulting optimization problem can be represented in the abstract form
\begin{subequations} \label{eq:abmm}
	\begin{alignat}{2}
		& \underset{x_1, \ldots, x_N}{\textrm{minimize }}~ &&  
			\Big\lVert \sum\nolimits_{i \in \agents} D_i x_i - d \Big\rVert  ^2 \\
		& \textrm{subject to}~~~ && E_i x_i = e_i, ~ \forall i \in \agents \label{eq:abmm_eq} \\
		& && M_i x_i \leq m_i, ~ \forall i \in \agents \label{eq:abmm_ineq} \\
		& && \sum\nolimits_{i \in \agents} N_i x_i = 0 \label{eq:abmm_couple},
	\end{alignat}
\end{subequations}
where \( x_1, \ldots, x_N \) are the agent-specific decision variables,
while \( d, e_i, m_i \) and \( D_i, M_i, N_i \) are vectors and matrices of appropriate dimensions.
A detailed and centralized version of \eqref{eq:abmm} can be found in \cite{spa_dvpp_placeholder}.
The objective function corresponds to the original problem where
\( \sum_{i \in \agents} D_i x_i \) and \( d \) represent the finite-horizon aggregate and desired
impulse response, respectively.
Constraints \eqref{eq:abmm_eq}, \eqref{eq:abmm_ineq} express state, input and output constraints 
of the system under a collection of worst-case disturbances;
\eqref{eq:abmm_eq} also includes the SLS affine constraints.
Finally, coupling constraints among the devices are prescribed in \eqref{eq:abmm_couple}.

In this work, we aim at solving \eqref{eq:abmm} under the assumption that agents do not share
their individual problem data and only peer-to-peer communication is possible.
Formally, the communication network is described by
the graph \( \graphcom = (\agents, \edges) \), 
where nodes correspond to agents
and
\( \edges \subseteq \agents \times \agents \)
is the set of edges.
The pair \( (i,j) \subseteq \agents \times \agents \)
belongs to \( \edges \) if and only if
agents \( i \) and \( j \) can directly communicate.
We denote \( \agents_i := \{j \in \agents ~|~ (i,j) \in \edges \} \) the set and \( d_i := | \agents_i| \) the number
of neighbors of agent \( i \in \agents \).
\section{Main Results} \label{sec:dual_consensus}
In this section, we develop a distributed algorithm to solve a generalization of \eqref{eq:abmm} 
and provide a convergence certificate.
The algorithm and results are then specialized to \eqref{eq:abmm}.
\subsection{Dual Consensus ADMM}
For the sake of generality, we consider the problem
\begin{equation} \label{eq:admm_primal}
	\begin{alignedat}{2}
		& \underset{x_1, \ldots, x_N}{\textrm{minimize }}~ &&  
		\sum\nolimits_{i \in \agents}^{} f_i(x_i) + g\left( \sum\nolimits_{i \in \agents} Q_i x_i \right) \\
	\end{alignedat}
\end{equation}
where \( x_i \in \setR^{n_i}, Q_i \in \setR^{m \times n_i} \)
and let \( \boldsymbol{x} := [x_1; \ldots; x_N] \in \setR^n \) with \( n := \sum_{i \in \agents} n_i \).
Problem \eqref{eq:abmm} is a special case  of \eqref{eq:admm_primal} 
under the substitution
\( 	f_i  := \indic_{\feas_i} \label{subeq:local_feas1},~
	\feas_i  := \{ x_i \in \setR^{n_i} ~|~ E_i x_i = e_i, ~ M_i x_i \leq m_i \},~
	Q_i  := [D_i;N_i],~
	g\left( [z_1; z_2] \right)  := \norm{z_1 - d}^2 + \indic_{\{0\}}(z_2), \)
where the dimensions of vectors \( z_1, z_2 \) correspond to the
number of rows of  \( D_i \) and \( N_i \), respectively. 
We will study \eqref{eq:admm_primal} under the following assumption.
\begin{assumption} \label{ass:admm}
	\
	\begin{enumerate}
		\item \( g \in \Gamma_m \) and \( f_i \in \Gamma_n \), for all \( i \in \agents \).
		\item A primal-dual solution exists and strong duality holds.
		\item The graph \( \graphcom \) is connected and undirected.
		\hfill \( \square \)
	\end{enumerate}
\end{assumption}
For our problem \eqref{eq:abmm}, 
feasibility corresponds to stabilizability of each system by \cite[Lemma 4.2]{sls} and
implies Assumption 1.1 and 1.2, with the latter following by Slater's constraint qualification \cite[\S 5.2.3]{boyd_convex}
since all constraints are affine.

Solving \eqref{eq:admm_primal} with a distributed algorithm is challenging due to the 
coupling introduced by \( \sum\nolimits_{i \in \agents} Q_i x_i \).
In the sequel, we will solve the dual of \eqref{eq:admm_primal} using the ADMM-based distributed algorithm proposed
in \cite{consensus_admm}, and then recover a primal solution.
The method in \cite{consensus_admm} addresses the \textit{decomposed consensus} optimization problem 
\begin{equation} \label{eq:consensus}
	\begin{alignedat}{2}
		& \underset{y}{\textrm{minimize }} ~ && \sum\nolimits_{i \in \agents} \zeta_i(y) + \xi_i(y)
	\end{alignedat}	
\end{equation}
where \( \zeta_i, \xi_i \in \Gamma_m \) for all \( i \in \agents \).
The relevant consensus ADMM algorithm is outlined in \autoref{alg:decomposed_consensus}.
Intuitively, \autoref{alg:decomposed_consensus} is derived by applying ADMM to \eqref{eq:consensus}
while allowing agents to have local estimates of \( y \) and enforcing the estimates of neighboring
agents to coincide (see \cite[App. A.2]{dual_consensus_large} for a detailed derivation).
We employ this algorithm because the dual of \eqref{eq:admm_primal}
is a decomposed consensus problem, as we show in the subsequent analysis.
Our approach resembles and extends those in \cite{dual_consensus_zero}, \cite{dual_consensus_large},
where the authors respectively address the cases where \( g = \indic_{\{0\}} \) and \( g = \indic_{\mathcal{K}} \),
for some nonempty, closed and convex cone \( \mathcal{K} \).

To derive its dual, we rewrite \eqref{eq:admm_primal} as 
\begin{equation} \label{eq:admm_primal_eq}
	\begin{alignedat}{2}
		& \underset{w,x_1, \ldots, x_N}{\textrm{minimize }}~ &&  
		\sum\nolimits_{i \in \agents}^{} f_i(x_i) + g\left( w \right) \\
		& \textrm{subject to}~~~ && \sum\nolimits_{i \in \agents} Q_i x_i = w
	\end{alignedat}
\end{equation}
and the Lagrangian function for this problem reads
\begin{align*}
	\lagrangian(\boldsymbol{x},w,y) & = \sum_{i \in \agents} f_i(x_i) + g(w) + y^T \bigg(\sum_{i \in \agents} Q_i x_i - w \bigg) 
\end{align*}
where \( y \in \setR^m \) is the dual variable.
The dual function is
\begin{align*}
	h(y) & = \inf_{x,w} \lagrangian(x,w,y) \\
	& = \inf_{x} \Big\{ \sum_{i \in \agents} \left( f_i(x_i) + y^T Q_i x_i  \right) \Big\}
	+ \inf_{w} \left\{ g(w) - y^T w  \right\} \\
	& = - \sum\nolimits_{i \in \agents} \conj{f}_i( -Q_i^T y) - \conj{g}(y)
\end{align*}
thus giving rise to the dual problem
\begin{equation} \label{eq:admm_dual}
	\begin{alignedat}{2}
		& \underset{y}{\textrm{maximize }}~ && - 
		\sum\nolimits_{i \in \agents} \left( \conj{f}_i( -Q_i^T y) + \frac{1}{N} \conj{g}(y) \right)
	\end{alignedat},
\end{equation}
which clearly belongs to the family of \eqref{eq:consensus}.
Consequently, \autoref{alg:decomposed_consensus} is applicable
with \( \zeta_i := \conj{f}_i \circ ( -Q_i^T ) \) and
\( \xi_i := \frac{1}{N} \conj{g} \).
\begin{algorithm}
	\caption{Decomposed Consensus ADMM}
	\begin{algorithmic}[1]
		\State \textbf{choose} \( \sigma, \rho > 0 \)
		\State \textbf{initialize} for all \( i \in \agents \):
		\(  p_i^0 = 0,  y_i^0, z_i^0, s_i^0 \in \setR^m \)
		\Repeat: for all \(  i \in \agents \)
		\State Exchange \( y_i^k \) with neighbors \( \agents_i \) 
		\State \( p_i^{k+1} \leftarrow p_i^k + \rho \sum_{j \in \agents_i} \left( y_i^k - y_j^k \right)  \) 
		\State \( s_i^{k+1} \leftarrow s_i^k + \sigma (y_i^k - z_i^k) \)
		\State \label{step:consensus_zeta} 
		\( y_i^{k+1} \leftarrow \argmin_{y_i} \bigl\{ \zeta_i(y_i) + y_i^T(p_i^{k+1} + s_i^{k+1}) \bigr. \hfill \) 
		\( ~~~~~~~~~~~~~~~~~~~~~ 
		\bigl. +\frac{\sigma}{2} \norm{y_i - z_i^k}^2 + \rho \sum_{j \in \agents_i} \norm{y_i - \frac{y_i^k + y_j^k}{2}}^2  \bigr\} \)
		\State 	\label{step:consensus_xi}
		\(  z_i^{k+1} \leftarrow \argmin_{z_i} \{ \xi_i(z_i) - z_i^T{s_i^{k+1}} + \frac{\sigma}{2} \norm{z_i - y_i^{k+1}}^2 \} \) 
		\Until{termination criterion is satisfied}
	\end{algorithmic}
	\label{alg:decomposed_consensus}
\end{algorithm}

Observe that, the functions \( \conj{f}_i, \conj{g} \) may not be known in closed form,
hence complicating steps \ref{step:consensus_zeta} and \ref{step:consensus_xi}
of \autoref{alg:decomposed_consensus}
which would require solving nested optimization problems.
To overcome this, we will reformulate these update rules in terms of the original
functions.
By completing the square, we can equivalently rewrite step \ref{step:consensus_xi} as
\begin{equation} \label{eq:proximal_reform}
	\begin{alignedat}{1}
		& \argmin\nolimits_{z_i} \Big\{ \frac{1}{N} \conj{g}(z_i) + \frac{\sigma}{2} \norm{z_i - \frac{s_i^{k+1} + \sigma y_i^{k+1}}{\sigma}}^2
		\Big\} \\
		& = \proxim{\conj{g}}^{N \sigma} \Big( \frac{s_i^{k+1}}{\sigma} + y_i^{k+1} \Big) \\
		& = \frac{s_i^{k+1}}{\sigma} + y_i^{k+1} - \frac{1}{N \sigma} \proxim{g}^{1/(N \sigma)} \Big(N  \big(s_i^{k+1} + \sigma y_i^{k+1}\big) \Big)
	\end{alignedat}
\end{equation}
where the last equality holds since 
\( x = \proxim{g}^{1/\gamma}(x) + \gamma \proxim{\conj{g}}^{\gamma}(x/\gamma) \)
by \cite[Th. 14.3(ii)]{convex_monotone2017}.

Similarly, square completion for step \ref{step:consensus_zeta} yields
\begin{equation*}
	\argmin_{y_i} \Big\{ \conj{f}_i(-Q_i^T y_i) + \frac{\sigma + 2 \rho d_i}{2}  \big\lVert y_i - \frac{1}{\sigma + 2 \rho d_i} r_i^{k+1} 
	\big\rVert ^2 \Big\}
\end{equation*}
where we defined for brevity \( r_i^{k+1} := \rho \sum_{j \in \agents}(y_i^k + y_j^k) + \sigma z_i^k - p_i^{k+1} - s_i^{k+1}  \).
Observe that the minimizer of the previous problem equals
\( \prox_{f_i^{\star} \circ \left( - Q_i^T \right)}^{\sigma + 2 \rho d_i} \left(\frac{r_i^{k+1}}{\sigma + 2 \rho d_i} \right) \)
and using \cite[Lemma B.1]{dual_consensus_large} with \( \mathcal{C} := \setR^m \) the previous problem admits the solution
\begin{equation} \label{eq:new_y_update}
	 y_i^{k+1} \leftarrow \frac{1}{\sigma + 2 \rho d_i} \left ( Q_i x_i^{k+1} +  r_i^{k+1} \right )
\end{equation}
introducing the auxiliary variable
\begin{equation} \label{eq:x_update}
	x_i^{k+1} \in \argmin_{x_i} \Big\{ f_i(x_i) + \frac{1}{2 (\sigma + 2 \rho d_i)} \norm{Q_i x_i + r_i^{k+1}}^2 \Big\}.
\end{equation}
Note that \eqref{eq:x_update} is guaranteed to be feasible but may, in general, admit an unbounded solution set.
This can lead to the situation described in the introduction, where the primal variables have no limit points 
and instead diverge to infinity.
To alleviate this, we introduce the artificial constraint \( \norm{x_i}_{\infty} \leq M_i \) for \( M_i > 0 \).
The parameters \( M_i \) need not be known in advance since we may increase \( M_i \) 
whenever the additional constraint renders the problem 
infeasible or the resulting solution is suboptimal for the original problem,
as will be shown in \autoref{alg:dual_conensus_admm}.
To do so, we solve \eqref{eq:x_update} both with and without the artificial constraint \( \norm{x_i}_\infty \leq M_i \) 
and compare the results.  
If \eqref{eq:x_update} with the artifical constraint is infeasible or has higher cost than \eqref{eq:x_update} without 
the artificial constraint, then we know we have not yet found an optimal solution to \eqref{eq:x_update}, 
so we double \( M_i \) and repeat this procedure.  
The result is that the solution of \eqref{eq:x_update} with the artifical constraint will converge to an optimal 
solution of \eqref{eq:x_update} without this artifical constraint exponentially fast.
This observation along with the following assumption will be used to show 
boundedness of \( x^{k}_i \) for all \( i \in \agents, k \in \setN \)
in our convergence analysis.
Then, we will be able to show that the limit set of the primal variables is nonempty, 
and that the primal variables converge to this limit set, which will guarantee convergence of our algorithm.
\begin{assumption} \label{ass:minim_contin}
	For each \( i \in \agents \) and \( \tilde{r}_i \in \setR^m \) there exist \( R > 0 \) and \( M_i \) finite, such that 
	\( \norm{r_i^{k+1} - \tilde{r}_i} \leq R \) implies that 
	the solution set of \eqref{eq:x_update} intersects the
	set \( \{x_i \in \setR^{n_i} ~|~ \norm{x_i}_{\infty} \leq M_i \}\).
	\hfill \( \square \)
\end{assumption}

Assumption \ref{ass:minim_contin} is fairly weak and holds for a large class of problems, including \eqref{eq:abmm}.

Applying \eqref{eq:proximal_reform}-\eqref{eq:x_update} to \autoref{alg:decomposed_consensus}
yields \autoref{alg:dual_conensus_admm},
where subproblems have been reformulated in terms of the original functions
and, therefore, the structure of the primal \eqref{eq:admm_primal} is preserved.
\autoref{alg:dual_conensus_admm} also includes the additional infinity norm constraints.
Note that, if the solution set of \eqref{eq:x_update} is bounded
the additional step \ref{step:x_update} can be disregarded.
\begin{algorithm}
	\caption{Dual consensus ADMM for \eqref{eq:admm_primal}.}
	\begin{algorithmic}[1]
		\State \textbf{choose} \( \sigma, \rho > 0 \)
		\State \textbf{initialize} for all \( i \in \agents \):
		\(  p_i^0 = 0, ~ y_i^0, z_i^0, s_i^0 \in \setR^m \)
		\Repeat: for all \(  i \in \agents \):
		\State exchange \( y_i^k \) with neighbors \( \agents_i \) 
		\State \( p_i^{k+1} \leftarrow p_i^k + \rho \sum_{j \in \agents_i} \left( y_i^k - y_j^k \right)  \) 
		\State \( s_i^{k+1} \leftarrow s_i^k + \sigma (y_i^k - z_i^k) \) \label{step:s_update}
		\State \( r_i^{k+1} \leftarrow \rho \sum_{j \in \agents_i}(y_i^k + y_j^k) + \sigma z_i^k - p_i^{k+1} - s_i^{k+1}  \)
		\State \label{step:x_preupdate}Compute the optimal value of \eqref{eq:x_update} 
		\State \label{step:x_update}
		\( x_i^{k+1} \leftarrow \text{Update using \eqref{eq:x_update}} \text{ with } \norm{x_i}_{\infty} \leq M_i. \hfill \)
		~~~~~~~~~~~~~~~~~If infeasible or suboptimal: \( M_i \leftarrow 2 M_i \), repeat
		Step \ref{step:x_update}.
		\State \label{step:y_update}
		\( y_i^{k+1} \leftarrow \text{Update using \eqref{eq:new_y_update}} \)
		\State \label{step:z_update}
		\( z_i^{k+1} \leftarrow \text{Update using \eqref{eq:proximal_reform}} \)
		\Until{termination criterion is satisfied}
	\end{algorithmic}
	\label{alg:dual_conensus_admm}
\end{algorithm}
\subsection{Distributed Optimal Control Design}
In this subsection, we specialize \autoref{alg:dual_conensus_admm} to \eqref{eq:abmm}
and exploit its structure to derive more efficient iterations.
In particular, using \( f_i = \indic_{\feas_i} \)
where \( \feas_i \) are polyhedral, we recognize that the problems
in steps \eqref{step:x_preupdate} and \eqref{step:x_update} are QPs.
Similarly, the proximal involved in step \ref{step:z_update} can be computed
in closed form
\begin{equation*}
	\proxim{g}^{1/(N \sigma) }([z_1;z_2]) = \bigg[\frac{2 N \sigma}{2 N \sigma + 1}d + \frac{1}{2 N \sigma +1} z_1;0 \bigg].
\end{equation*}
Thus, our algorithm requires each agent to solve just two QPs per iteration
for large enough \( M_i \).

\subsection{Convergence Analysis}
Next, we will establish convergence of \autoref{alg:dual_conensus_admm} and 
show how a primal solution can be recovered.
Firstly, we recall several useful statements from convex analysis.
\begin{lemma} \label{lemma:fermat}
	Let \( y,z \in \setR^n, f \in  \Gamma_n \) and consider the sequence
	\( (x^k, u^k)_{k \in \setN} \) such that \(  x^k \to x, u^k \to u \)
	and \( u^k \in \partial f(x^k) \) for all \( k \in \setN \). Then,
	\begin{enumerate}
	\item \( 	\argmin f = \{x \in \setR^n ~|~ 0 \in \partial f(x)\} \) \cite[Th. 16.3]{convex_monotone2017},
	\item \( y \in \partial \conj{f}(z) \iff z \in \partial f(y) \) \cite[Cor. 16.30]{convex_monotone2017},
	\item \( u \in \partial f(x) \) \cite[Cor. 16.36]{convex_monotone2017}.
	\hfill \( \square \)
	\end{enumerate}
\end{lemma}
We will employ the previous lemma to prove our main result.
\begin{theorem}
	The iterates \( ((y_i^{k})_{i \in \agents})_{k \in \setN}, ((z_i^{k})_{i \in \agents})_{k \in \setN} \) 
	of \autoref{alg:dual_conensus_admm}
	converge to a maximizer of the dual \eqref{eq:admm_dual},
	while \( ((x_i^{k})_{i \in \agents})_{k \in \setN} \) converges to the set of minimizers of the primal \eqref{eq:admm_primal}.
	\hfill \( \square \)
\end{theorem}
\begin{proof}
	Initially, we recall that \autoref{alg:dual_conensus_admm} is derived by applying ADMM to 
	\eqref{eq:consensus} and that ADMM is a special case of the Douglas-Rachford splitting algorithm \cite{admm_drs}.
	Invoking \cite[Cor. 28.3]{convex_monotone2017} it holds
	\( 	y^k_i \to y^{\star}, ~ z^k_i \to y^{\star} \) for all \( i \in \agents \),
	where \( y^{\star} \) is a maximizer of \eqref{eq:admm_dual}.
	Further, we note that the first-order optimality conditions for a primal-dual solution of \eqref{eq:admm_primal} 
	correspond to saddle-points of the Lagrangian and, hence, are given by
	\begin{subequations} \label{eq:optimality}
		\begin{align} 
			0 & \in \partial_{x_i} \lagrangian(x,w,y) = \partial f_i(x_i) + Q_i^T y, ~\forall i \in \agents
			\label{eq:optimality_1} \\
			0 & \in \partial_{w} \lagrangian(x,w,y) = \partial g(w) - y \label{eq:optimality_2} \\
			0 & = \grad_y \lagrangian(x,w,y) = \sum\nolimits_{i \in \agents} Q_i x_i - w. \label{eq:optimality_3}
		\end{align} 
	\end{subequations}
	We denote \( \boldsymbol{x}^k := (x^k_i)_{i \in \agents}, \boldsymbol{x}^{\star} \in \omega(\boldsymbol{x}^k), 
	w^{k} := \sum_{i \in \agents} s_i^{k} \) and \( w^{\star}:=\lim_{k \to \infty}w^k \) which is well-defined by
	step \ref{step:s_update} since \( y^k_i, z^k_i \to y^{\star} \).
	In the sequel, we will show that \( (\boldsymbol{x}^{\star}, y^{\star}, w^{\star}) \) satisfies \eqref{eq:optimality}
	which implies that \( \boldsymbol{x}^{\star} \) is a primal minimizer for every 
	\( \boldsymbol{x}^{\star} \in \omega(\boldsymbol{x}^k) \).
	
	Consider the minimization involved in step \ref{step:x_update} and note
	that by the choice of \( x^{k+1} \) in step \ref{step:x_update} it is a minimizer of \eqref{eq:x_update}.
	Employing \autoref{lemma:fermat}.1 yields
	\begin{align*}
		0 & \in \partial f_i(x_i^{k+1}) + \frac{1}{\sigma + 2 \rho d_i} Q_i^T \left( Q_i x_i^{k+1} + r_i^{k+1} \right) \\
		& = \partial f_i(x_i^{k+1}) + Q_i^T y_i^{k+1}
	\end{align*}
	which corresponds to \eqref{eq:optimality_1} being satisfied in each iteration and,
	by \autoref{lemma:fermat}.3, \eqref{eq:optimality_1} is satisfied by \( (\boldsymbol{x}^{\star}, y^{\star}) \).
	Next, invoking \autoref{lemma:fermat}.1 for step \ref{step:z_update} in the form of
	\eqref{eq:proximal_reform}, summing over all \( i \in \agents \) and rearranging terms we obtain
	\begin{align*}
		w^{k+1} - \frac{1}{N} \sum_{i \in \agents} \sigma (z_i^{k+1} - y_i^{k+1} ) \in \frac{1}{N} \sum_{i \in \agents} \partial \conj{g}(z_i^{k+1}).
	\end{align*}
	Then, taking the limit and utilizing \autoref{lemma:fermat}.3 results in
	\begin{align*}
		\lim_{k \to \infty} w^{k+1} \in \frac{1}{N} \sum_{i \in \agents}  \partial \conj{g}\Big(\lim_{k \to \infty} z_i^{k+1}\Big)
		& \implies w^{\star} \in \partial \conj{g}(y^{\star})  \\
		& \implies y^{\star} \in \partial g(w^{\star})
	\end{align*}
	where the first implication stems from \( z^k_i \to y^{\star} \) while the second follows 
	from \autoref{lemma:fermat}.2; thus, showing \eqref{eq:optimality_2}.
	
	Finally, substituting \( r_i^{k+1} \) in step \ref{step:y_update} and summing over
	all \( i \in \agents \) it holds
	\begin{align*}
		\sum_{i \in \agents} Q_i x_i^{k+1} - w^{k+1} & = \sigma \sum_{i \in \agents} (y_i^{k+1} - z_i^k)
		+ \sum_{i \in \agents} p_i^{k+1}  \\
		& ~~~~~ + \rho \sum_{i \in \agents} \sum_{j \in \agents_i} (2 y_i^{k+1} - y_i^k - y_j^k)
	\end{align*}
	where \( \sum_{i \in \agents} p_i^{k+1} = 0 \) for all \( k \in \setN \), by properties of the consensus ADMM algorithm
	\cite[App. A.1]{dual_consensus_large}.
	Hence, we see that the right-hand side vanishes as \( k \to \infty \) thus implying
	\( \sum_{i \in \agents} Q_i x_i^{\star} = w^{\star}\),
	which amounts to \eqref{eq:optimality_3}.
	Therefore, we established that \( (\boldsymbol{x}^{\star}, y^{\star}, w^{\star}) \) satisfies 
	\eqref{eq:optimality}.
	
	To conclude the proof, it suffices to show that \( \boldsymbol{x}^k \to \omega(\boldsymbol{x}^k) \)
	since every \( \boldsymbol{x}^{\star} \in \omega(\boldsymbol{x}^k) \) is primal optimal.
	Since \( y_i^k, z_i^k \to y^{\star} \) and \( p_i^k \) and \( s_i^k \) also converge to limits, 
	we have that \( r_i^k \) converges to a limit, say \( r_i^{\star} \),
	and by definition there exists \( \overline{N} \) finite such that \( k \geq \overline{N} \) implies
	\( \norm{r_i^k - r_i^{\star}} \leq R \) for all \( i \in \agents \).
	Therefore, by Assumption \autoref{ass:minim_contin} and step \ref{step:x_update}, we deduce that \( (\boldsymbol{x}^k)_{k \geq \overline{N}} \) is bounded.
	As \( \overline{N} \) is finite, \( (\boldsymbol{x}^k)_{k \in \setN} \) is also bounded and thus 
	contained in a compact set,
	which implies that \( \omega(\boldsymbol{x}^k) \) is nonempty and
	\( \boldsymbol{x}^k \to \omega(\boldsymbol{x}^k) \) by virtue of \cite[App. A.2]{khalil1996nonlinear},
	hence concluding the proof.
\end{proof}
\begin{remark}
	The prior works \cite{dual_consensus_zero} and \cite{dual_consensus_large} 
	ensure primal optimality of any limit point of \( ((x_i^{k})_{i \in \agents})_{k \in \setN} \)
	for their algorithms;
	nonetheless, as mentioned above limit points are not guaranteed to exist for their algorithms, 
	so the set of primal solutions can diverge to infinity in those cases.
	To see this, recall that the optimization problem in update rule \eqref{eq:x_update} need not be 
	strongly convex, e.g., when \( Q_i \) has a nontrivial null space, which implies that
	\( x_i^{k+1} \) can diverge to infinity.
	Our Assumption \autoref{ass:minim_contin} along with step \ref{step:x_update} of \autoref{alg:dual_conensus_admm}
	preclude this behavior and establish the existence of limit points of \( ((x_i^{k})_{i \in \agents})_{k \in \setN} \),
	thus ensuring convergence of the primal variables to the set of primal minimizers.
\end{remark}

\section{Test Case} \label{sec:test_case}
To demonstrate the effectiveness of our proposed control design, we will consider
the control of an aggregation of DERs, referred to as virtual power plant (VPP), 
including wind turbines (WTs), photovoltaics (PVs) and energy storage devices (ESs),
that collectively provide fast frequency regulation to the power grid.

\begin{figure*}[!t]
	\vspace{5pt}
	\begin{subfigure}{0.33\textwidth}
		\centering
		\includegraphics[width=0.95\linewidth]{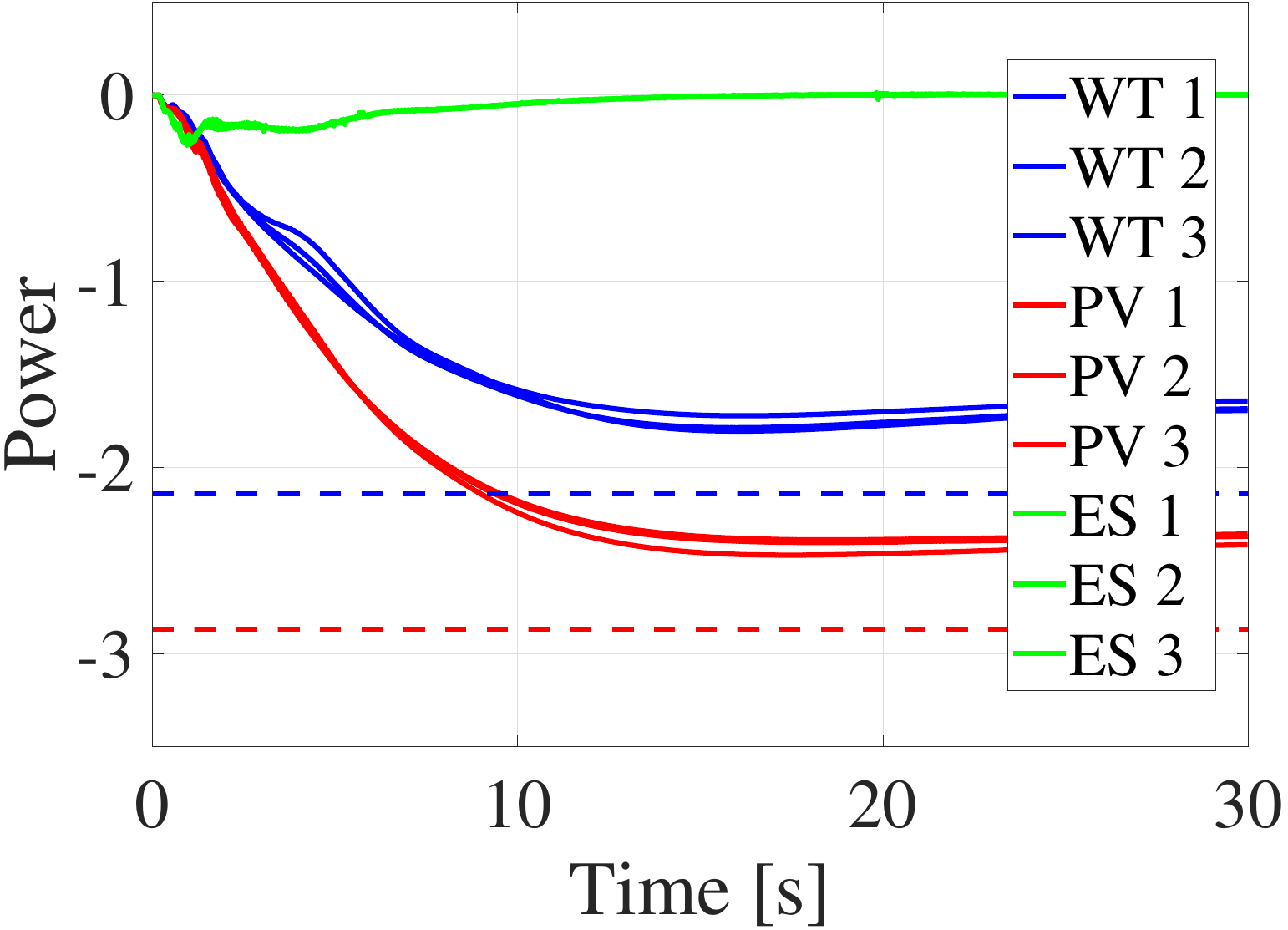}
		\caption{Scenario 1}
		\label{subfig:perfect_dec}
	\end{subfigure}
	\begin{subfigure}{0.33\textwidth}
		\centering
		\includegraphics[width=0.95\linewidth]{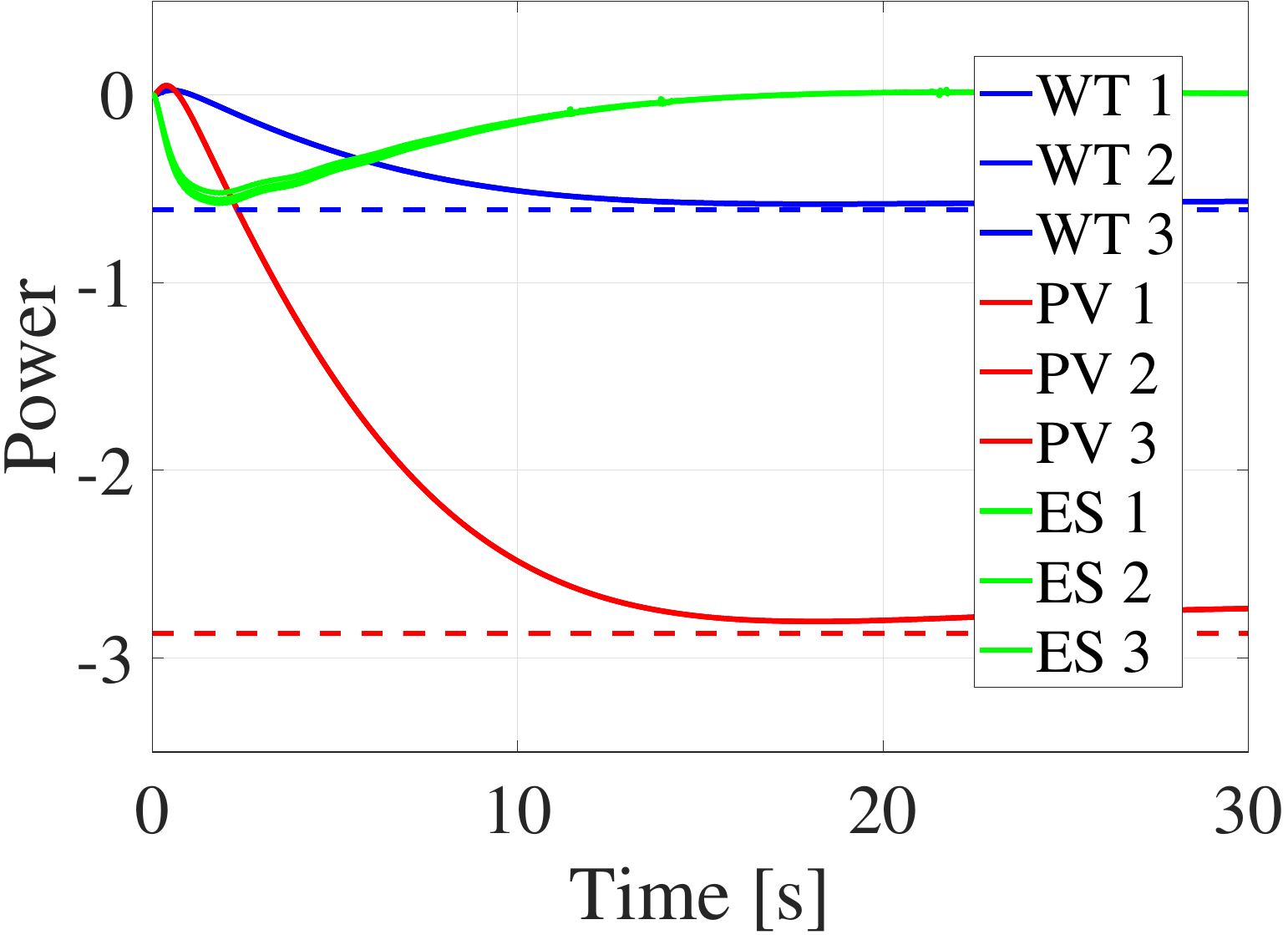}
		\caption{Scenario 2}
		\label{subfig:imperfect_dec}
	\end{subfigure}
	\hfill
	\begin{subfigure}{0.33\textwidth}
		\centering
		\includegraphics[width=0.95\linewidth]{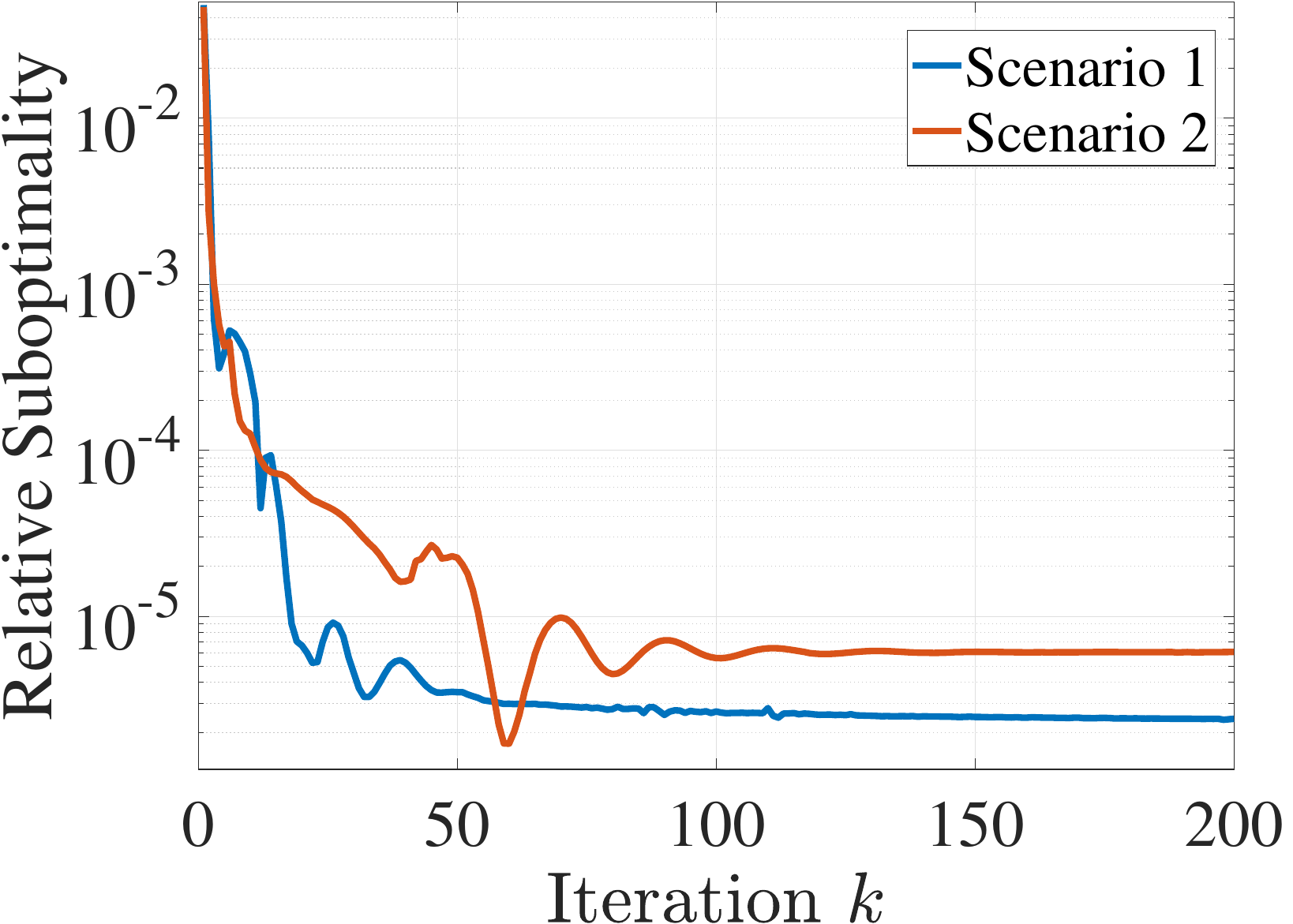}
		\caption{Relative Suboptimality}
		\label{subfig:suboptimality}
	\end{subfigure}
	\caption{(a), (b) Power output of each device \( P_i^k \), as deviation from its setpoint.
		Dashed lines correspond to power constraints \(  |P^k_i| \leq \overline{P}_i \) for devices of the same color 
		(ES constraints exceed axis limits).
		(c) Convergence rate of \autoref{alg:dual_conensus_admm}.}
	\label{fig:powers}
\end{figure*}

\subsection{Power Grid Model}
We adopt a first-order system representation for both the grid frequency 
and the power of the grid of the form
\begin{align}
	\omega^{k+1} & = \omega^k + 0.083(P_{g}^k + P_{vpp}^k + P_{ext}^k) \\
	{P}_{g}^{k+1} & = 0.9944~P_{g}^k +0.0015~\omega^k
\end{align}
where \( \omega^k \) is the frequency of the grid,
\( P_{sg}^k, P_{vpp}^k \) denote the power of the
grid and the VPP,
whereas \( P_{ext}^k \) models external power injections or losses
and is a step function of size 70.
All variables correspond to deviations from setpoints.

\subsection{DER Model}
For individual DERs, we discretize the continuous-time model in \cite{verena_dvpp}
which results in the following discrete-time LTI description
\begin{equation*}
	\renewcommand*{\arraystretch}{0.9}
	\begin{aligned}
		x^{k+1}_i  &= \Vector{1 -  \frac{hk_{I}}{k_{P}} & 0 & 0 \\
			1 & 0 & 0 \\
			0 & 0 & \frac{\tau_i-h}{\tau_i} }
		x^k_i +	
		\Vector{0 \\ 0 \\ \frac{h}{\tau_i}} u^k
		+ \Vector{\frac{-h k_{I}}{k_{P}^2} \\ \frac{1}{k_{P}} \\ 0} \omega^k \\
		P^k_i & = \Vector{0 & 1025 & -143} x^k
	\end{aligned}
\end{equation*}
where \( k_I = 1700, k_P = 150 \) are constants and
\( h = 0.0167 \) is the discretization time step in seconds.
The parameter \( \tau_i \) is the time constant of each DER and is
equal to 1.3, 0.55 and 0.15 seconds for WTs, PVs and ESs, respectively.
For simplicity of exposition, parameters will be shared among DERs of the same type.
The power output of each device is \( P_i^k \) and it holds that
\( P_{vpp}^k = \sum_{i \in \agents} P_i^k. \)

Our considered VPP is composed of 3 WTs, 3 PVs
and 3 ESs communicating over a randomly-generated connected undirected graph.
We explicitly include communication and computation delays related to our distributed
algorithm in our simulations, by assuming that one iteration of \autoref{alg:dual_conensus_admm} is performed
every 6 time steps of the controlled system. 
This means that DER controllers are updated
significantly slower than the time scale of the faster system dynamics.

\subsection{Control Design Parameters}
The main component of our controller is the desired aggregate
behavior which we designate as a low-pass filter,
according to grid operator specifications, with time constant
0.1 and gain 53.1.
For \( \mathcal{P}_i \), we use 10 simple poles for each agent that include the poles of the desired transfer function,
the plant and the remaining are chosen along a spiral inside the unit disk as described in
\cite[Section \uppercase\expandafter{\romannumeral 4\relax}]{msls_placeholder}.
Engineering limitations on the device level are specified as upper
bounds on \( |P^k_i| \), denoted \( \overline{P}_i \),
under a worst-case frequency signal which we choose as a step input of size 0.25.
We set \( \overline{P}_i \) to 2.1, 2.9 and 27.5 for WTs, PVs and ESs, respectively,
which correspond to device-specific percentages of the nominal power \( \hat{P}_{i} \) of 
each device,
where the values for WTs and PVs are small because their nominal operation is close to their maximum power capacity.
We, additionally, constrain ESs to have zero output at steady-state, to avoid impacting their state of charge.
Finally, we enforce the contribution of WTs and PVs to the aggregate steady-state power output of the VPP to be proportional
to their nominal power, i.e., \( P_i^{\infty}/P_{vpp}^{\infty} = \hat{P}_{i} \), in order to achieve fair power sharing in
accordance with grid codes.
This specification is encapsulated as a coupling constraint of the form \eqref{eq:abmm_couple} and
ensures that the power deviation introduced to compensate \( P_{ext} \)
is distributed fairly among devices.

\subsection{Simulations}
In the following simulations, we utilize \autoref{alg:dual_conensus_admm} in an online fashion 
in the sense that device controllers are iteratively updated during system operation.
For comparison, we also demonstrate the response obtained via a centralized solution of \eqref{eq:abmm},
as well as that of the desired transfer function.
We set \( \rho = \sigma = 0.1 \) in \autoref{alg:dual_conensus_admm} for all simulations. 

In our first scenario, we deploy \autoref{alg:dual_conensus_admm} cold-started, i.e., with
\(  p_i^0 = y_i^0 = z_i^0 = s_i^0 = 0 \) for all \( i \in \agents \),
which corresponds to the challenging case of a disturbance occurring while local
controllers are being initialized.
We consider operating conditions where perfect model matching is feasible.
\autoref{subfig:perfect_dec} demonstrates the power output of each device
whereas \autoref{fig:perfect_all} compares the distributed, centralized and desired response 
for the aggregate power output \( P_{vpp}^k \).
Our distributed scheme achieves similar performance as the centralized approach,
with small oscillating during the initial transient (see \autoref{fig:perfect_all_detail}).
\begin{figure}[h]
	\vspace{5pt}
	\begin{subfigure}{0.24\textwidth}
		\centering
		\includegraphics[width=0.95\linewidth]{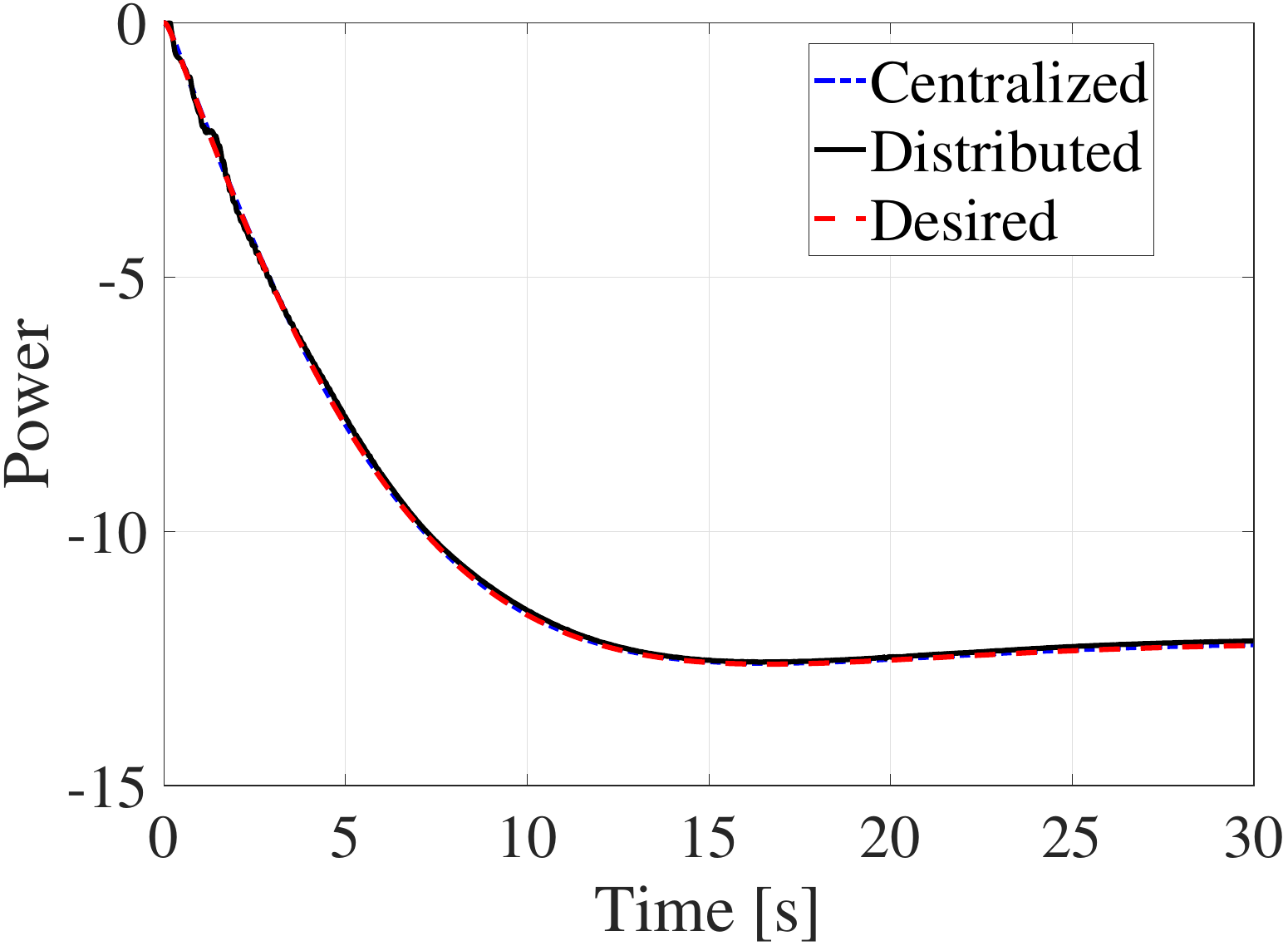}
		\caption{Scenario 1: Full simulation}
		\label{fig:perfect_all}
	\end{subfigure}
	\hfill
	\begin{subfigure}{0.24\textwidth}
		\centering
		\includegraphics[width=0.95\linewidth]{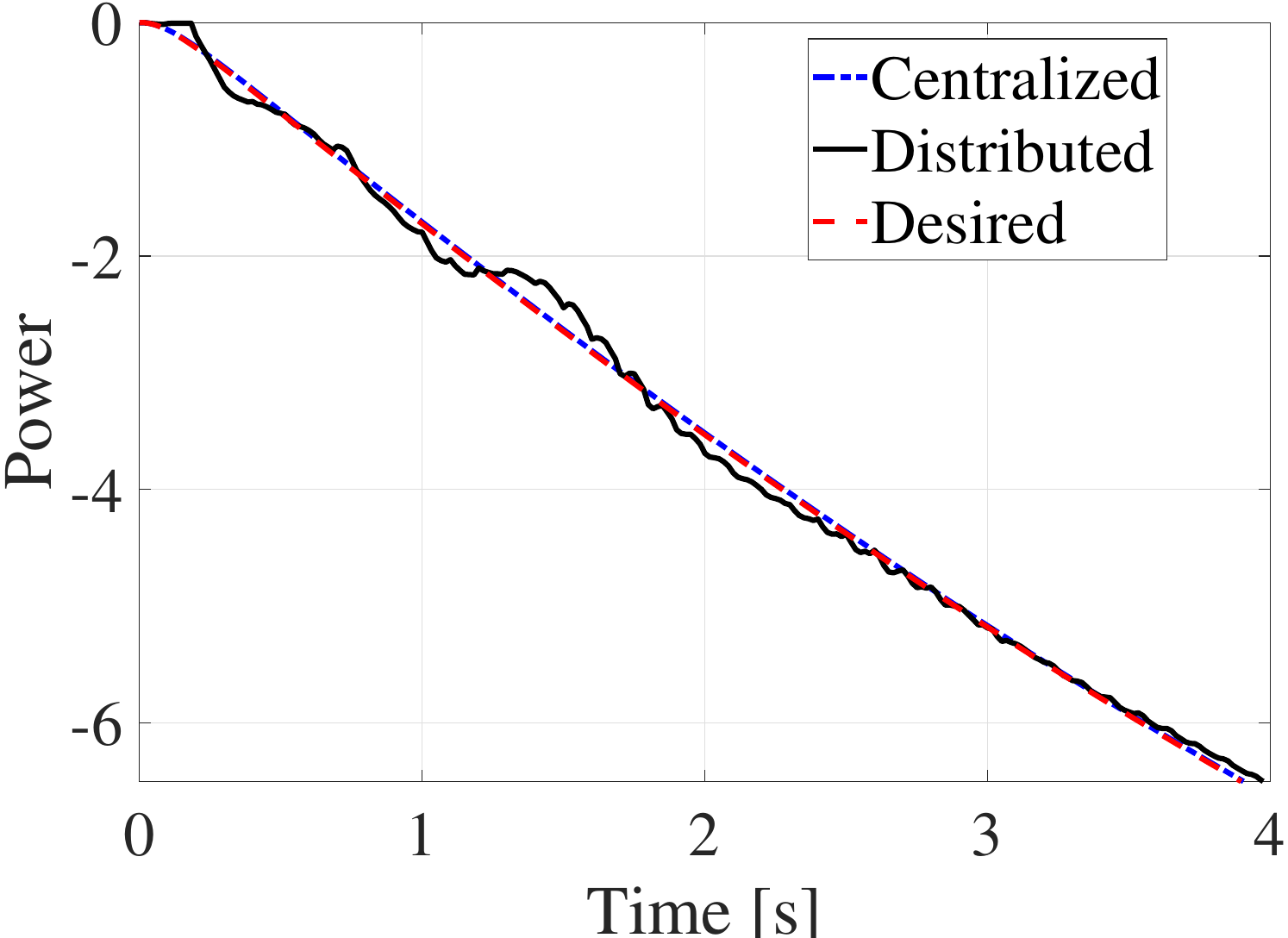}
		\caption{Scenario 1: Initial transient}
		\label{fig:perfect_all_detail}
	\end{subfigure}
	\begin{subfigure}{0.24\textwidth}
		\centering
		\includegraphics[width=0.95\linewidth]{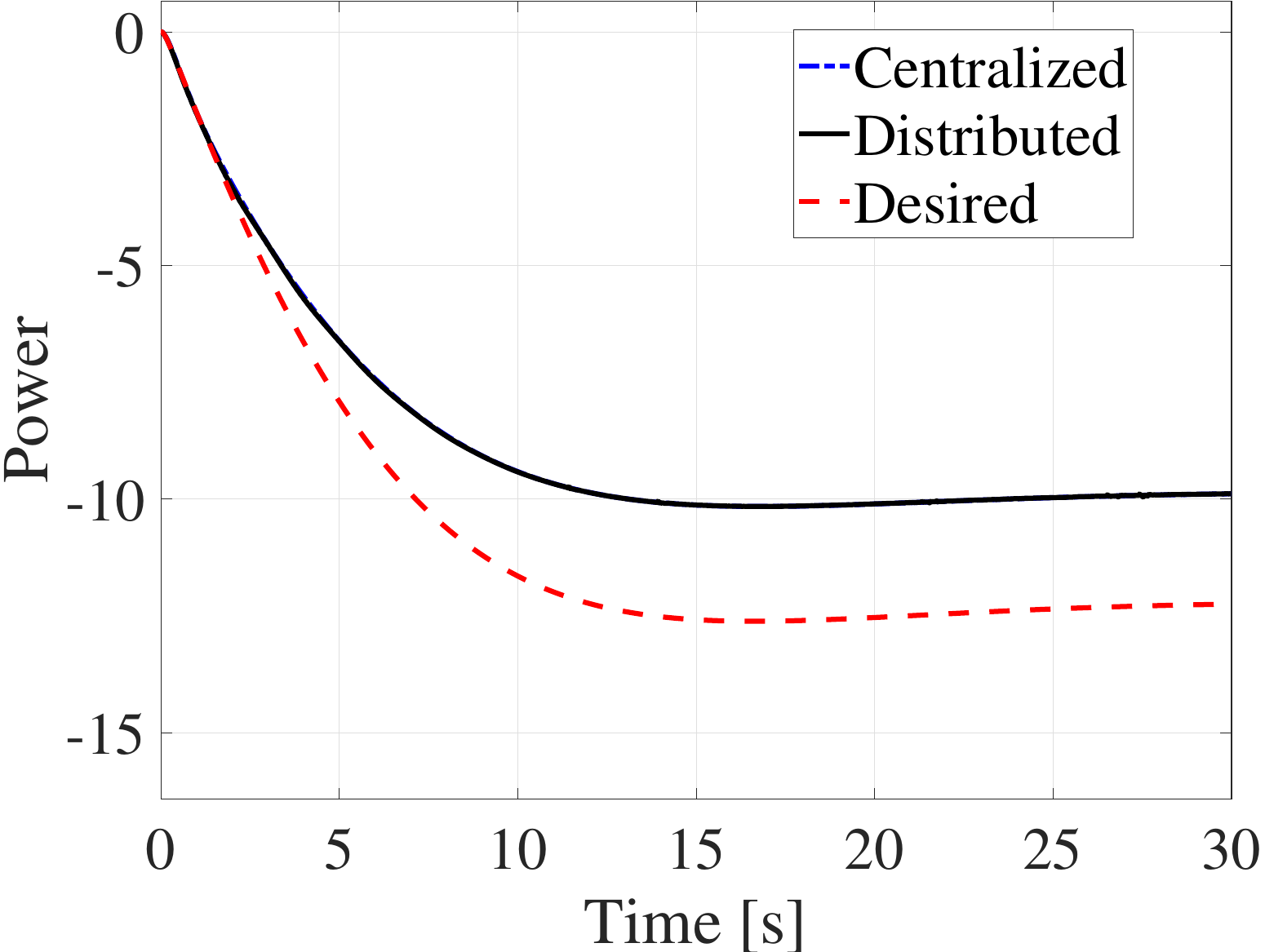}
		\caption{Scenario 2: Full simulation}
		\label{fig:imperfect_all}
	\end{subfigure}
	\hfill
	\begin{subfigure}{0.24\textwidth}
		\centering
		\includegraphics[width=0.95\linewidth]{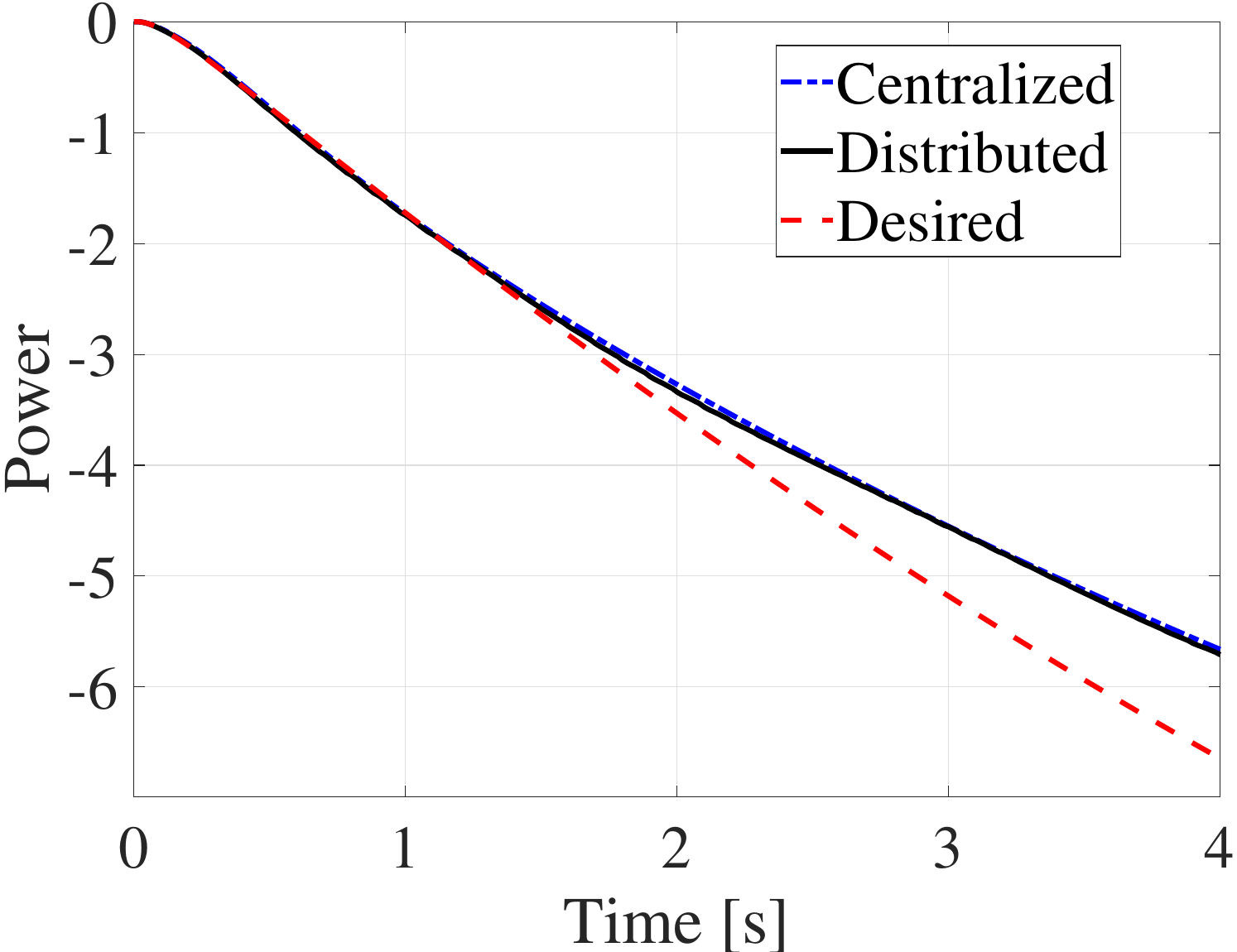}
		\caption{Scenario 2: Initial transient}
		\label{fig:imperfect_all_detail}
	\end{subfigure}
	\caption{Aggregate power response \( P_{vpp}^k \) in Scenario 1 and Scenario 2.}
	\label{fig:all_responses}
\end{figure}

In our second scenario, we consider a decrease of nominal power from WTs, e.g.,
due to a decrease in local wind speeds, thus tightening their power limit constraints
to 0.612 and making perfect model matching not feasible.
The individual and aggregate device output is shown
in \autoref{subfig:imperfect_dec} and \autoref{fig:imperfect_all}, respectively.
In this case, we perform 50 iterations of \autoref{alg:dual_conensus_admm} before
starting the simulation, which more closely corresponds to the nominal case
where local controllers are distributedly computed before deployment.
Performance is significantly improved during transient
where oscillations are limited, see 
\autoref{fig:perfect_all_detail} and \autoref{fig:imperfect_all_detail},
and steady-state where coupling constraints are more rigidly satisfied, see \autoref{subfig:imperfect_dec}.
This improvement is justified by the rapid convergence of \autoref{alg:dual_conensus_admm}, 
reaching a relative suboptimality of \( 10^{-5} \) w.r.t.\ the optimal value of \eqref{eq:abmm}
within 100 iterations, as shown in \autoref{subfig:suboptimality},
which is why the warm-started Scenario 2 shows improved performance over the cold-started Scenario 1.
Moreover, the reduced nominal power of the WTs leads to a decreased contribution in the
aggregate steady-state power output as prescribed by the coupling constraints,
while the opposite behavior is observed for PVs.
The ESs are more active in Scenario 2 and compensate for WTs,
especially during the initial transient given that ESs are characterized
by the smallest time constant.
Further, \autoref{subfig:imperfect_dec} clearly indicates that devices satisfy constraints
non-conservatively and are being operated at their limits.
Finally, our distributed method yields very similar responses
to the centralized solution, even under online deployment,
as shown in \autoref{fig:all_responses}.
\section{Conclusion} \label{sec:conclusion}
In this paper, we proposed a distributed, optimal control design of local state feedback
controllers for a set of interconnected agents with state and input constraints.
Our approach employs SLS with SPA to prescribe a desired aggregate behavior and impose
device constraints non-conservatively via a convex optimization problem,
without closed-loop finite impulse response or locality constraints.
The main challenge addressed was solving SLS with SPA in a distributed manner, 
for which we developed a dual consensus ADMM algorithm and, under weak assumptions, 
provided a convergence certificate that guarantees convergence to the set of primal minimizers, 
unlike similar dual consensus ADMM algorithms where divergence of the primal variables was possible.
Additionally, our optimization algorithm extended existing theoretical results
on DAO, by handling coupling constraints and lifting certain common, yet restrictive, regularity conditions
on the objective, e.g., strong convexity and differentiability.
We demonstrated the effectiveness of our method on a control design problem for 
DERs that collectively provide stability services to the power grid.
Future work includes investigating the convergence rate of the proposed
distributed optimization algorithm, and development of distributed solutions for 
SLS with SPA in the case of \( \mathcal{H}_{\infty} \) control synthesis.
\bibliographystyle{ieeetran}
\bibliography{sls_admm}

\end{document}